\documentclass[a4paper]{amsart}

\usepackage[english]{babel}
\usepackage[latin1]{inputenc}
\usepackage{enumerate} % til \begin{enumerate}[(i)]
\usepackage{latexsym} % symboler
\usepackage{amssymb} % flere symbloer
\usepackage{amsmath} % til pmatrix
\usepackage[all,arc,curve,color,frame]{xy}
\newtheorem{theorem}{Theorem}

\newtheorem{definition}{Definition}
\newtheorem{proposition}{Proposition}

\newtheorem{conjecture}{Conjecture}
\theoremstyle{remark}
\newtheorem{remark}{Remark}

\begin{document}

\title{On a counterexample to a conjecture by Blackadar}
\author{Adam P. W. S{\o}rensen}
\address{Department of Mathematical Sciences\\
University of Copenhagen\\
Universitetsparken 5, DK-2100, Copenhagen \O, Denmark}
\email{apws@math.ku.dk}
\date{\today}

\maketitle

\begin{abstract}
Blackadar conjectured that if we have a split short-exact sequence $0 \to I \to A \to \mathbb{C} \to 0$ where $I$ is semiprojective then $A$ must be semiprojective. 
Eilers and Katsura have found a counterexample to this conjecture.
Presumably Blackadar asked that the extension be split to make it more likely that semiprojectivity of $I$ would imply semiprojectivity of $A$.
But oddly enough, in all the counterexamples of Eilers and Katsura the quotient map from $A$ to $A/I \cong \mathbb{C}$ is split.
We will show how to modify their examples to find a non-semiprojective $C^*$-algebra $B$ with a semiprojective ideal $J$ such that $B/J$ is the complex numbers and the quotient map does not split.
\end{abstract}

\section{Introduction} \label{apws:sec:intro} 

Semiprojectivity is a lifting property for $C^*$-algebras. 
It was introduced in \cite{apws:BlackadarShape} in a successful attempt to transfer some of the power of shape theory for metric spaces to the world of $C^*$-algebras. 

\begin{definition}
A $C^*$-algebra $A$ is \index{semiprojective} semiprojective if whenever we have a $C^*$-algebra $B$ containing an increasing sequence of ideals $J_1 \subseteq J_2 \subseteq \cdots$, and a $*$-homomorphism $\phi \colon A \to B/ \overline{\cup_k J_k} $, we can find an $n \in \mathbb{N}$ and a $*$-homomorphism $\psi \colon A \to B/J_n$ such that 
\begin{equation*}
	\pi_{n,\infty} \circ \psi = \phi,
\end{equation*}
where $\pi_{n,\infty} \colon B/ J_n \twoheadrightarrow B/\overline{\cup_k J_k}$ is the natural quotient map.
\end{definition}

Pictorially, $A$ is semiprojective if we can always fill in the dashed arrow in the following commutative diagram:
\begin{equation*}
\xymatrix{
				& B \ar@{->>}[d] \\
				& B / J_n \ar@{->>}[d] \\
	A \ar[r]_-{\phi} \ar@{-->}[ur]^{\psi}	& B/ \overline{\cup_k J_k}.
}
\end{equation*}

The book \cite{apws:LoringBook} is the canonical source for information about semiprojectivity.
See also the more recent paper \cite{apws:BlackadarSimpleSp}, the beginning of which has an expository nature.  

Many of the main problems about semiprojectivity are concerned with the permanence properties of semiprojective $C^*$-algebras. 
In \cite{apws:BlackadarShape} Blackadar proves that the direct sum of two unital semiprojective $C^*$-algebras is again semiprojective, and that if $A$ is unital and semiprojective then $M_n(A)$ is also semiprojective. 
These results where later extended from unital algebras to $\sigma$-unital algebras, so in particular to all separable algebras, by Loring in \cite{apws:LoringDirectSum}.  
The results are a little stronger, in fact we have for separable algebras that $A \oplus B$ is semiprojective if and only if both $A$ and $B$ are, and a separable unital algebra $D$ is semiprojective if and only if $M_2(D)$ is. 
It is still an open problem if a non-unital $A$ must be semiprojective whenever $M_2(A)$ is.
It is true if $A$ is commutative, see \cite[Corollary 6.9]{apws:ST}. 

For a long time the following conjecture by Blackadar (\cite[Conjecture 4.5]{apws:BlackadarSimpleSp}), which was first asked as a question by Loring in \cite{apws:LoringBook}, was one of the main questions concerning the permanence properties of semiprojective $C^*$-algebras:

\begin{conjecture}[Blackadar] \label{apws:seqConj}
Let 
\begin{equation*}
	0 \to A \to B \to \mathbb{C} \to 0
\end{equation*}
be a split exact sequence of separable $C^*$-algebras. If $A$ is semiprojective then so is $B$. 
\end{conjecture}

An important partial result was obtained in \cite[Theorem 6.2.1]{apws:ELP}. 
It was used in \cite{apws:ELP} to show that all the so called one-dimensional non-commutativ CW complexes are semiprojective.  
Enders (\cite{apws:Enders}) has proved a form of converse to Conjecture \ref{apws:seqConj}, namely that if $0 \to A \to B \to \mathbb{C} \to 0$ is an exact sequence of separable $C^*$-algebras with $B$ semiprojective then $A$ is semiprojective.  

Recently Eilers and Katsura (\cite{apws:EilersKatsura}) have found a counterexample to Conjecture \ref{apws:seqConj}:

\begin{theorem}[Eilers-Katsura] \label{apws:EKseq}
There exists a split short exact sequence
\begin{equation*}
	0 \to A \to B \to \mathbb{C} \to 0
\end{equation*}
where $A$ is semiprojective but $B$ is not. 
\end{theorem}
 
The techniques used by Eilers and Katsura comes from the world of graph $C^*$-algebra, and so only leads to split short exact sequence.
Their work leaves open the question of whether there is a non-split short exact sequence $0 \to A \to B \to \mathbb{C} \to 0$ with $A$ semiprojective and $B$ not semiprojective. 
In light of Eilers and Katsura's result we certainly expect such a sequence to exist, and indeed, as we shall see in Theorem \ref{apws:mainThm}, it does.  

This note is structured as follows: In Section \ref{apws:sec:tool} we prove two propositions that will be our main tools, in Section \ref{apws:sec:counter} we prove the main theorem. 

\section{Toolbox} \label{apws:sec:tool}

We will be working with pullbacks. 
Given two $*$-homomorphisms $\phi \colon A \to D$, $\psi \colon B \to D$, we write, by standard abuse of notation, the pullback of $A$ and $B$ taken over $\phi$ and $\psi$ as $A \oplus_D B$. 
That is $A \oplus_D B = \{ (a,b) \in A \oplus B \mid \phi(a) = \psi(b) \}$. 
The pullback is universal for $*$-homomorphisms into $A$ and $B$ that agree after compositions with $\phi$ and $\psi$.
For a detailed account of the theory of pullbacks (and pushouts) see \cite{apws:pushpull}. 

Our first tool will let us produce new short exact sequences from old ones. 
In particular it gives us a way to alter a split short exact sequence to make it non-split. 

\begin{proposition} \label{apws:constructionNonSplit}
Suppose we are given two short exact sequence  
\begin{equation} \label{apws:acrossSeq}
	0 \to I \to A \stackrel{\pi}{\to} \mathbb{C} \to 0,
\end{equation}
and 
\begin{equation} \label{apws:downSeq}
	0 \to J \to B \stackrel{\rho}{\to} \mathbb{C} \to 0.
\end{equation}
Let $P$ be the pullback of $A$ and $B$ taken over $\pi$ and $\rho$. 
Then the following three sequences are short exact:
\begin{eqnarray}
	0 \to I \oplus J \to &P& \to \mathbb{C} \to 0, \label{apws:pullSeq} \\
	0 \to I \to &P& \to A \to 0, \quad \mathrm{ and,} \label{apws:newDownSeq}\\
	0 \to J \to &P& \to B \to 0. \label{apws:newAcrossSeq}
\end{eqnarray} 
Moreover (\ref{apws:pullSeq}) splits if and only if both (\ref{apws:acrossSeq}) and (\ref{apws:downSeq}) splits. 
\end{proposition}
\begin{proof}
We begin by proving that (\ref{apws:newDownSeq}) is exact. 
The map from $P$ to $A$ is simply projection onto the first coordinate, which is a surjection since both $\pi$ and $\rho$ are surjections.  
The kernel consists of pairs $(a,b) \in P$ with $a = 0$, that is pairs $(0,b)$ where $\rho(b) = 0$.
Hence the kernel is $0 \oplus I \cong I$.
A similar argument shows that (\ref{apws:newAcrossSeq}) is exact. 

We now consider (\ref{apws:pullSeq}). 
The map from $P$ to $\mathbb{C}$ takes a pair $(a,b)$ and sends it to $\pi(a) (= \rho(b))$. 
By the surjectivity of $\pi$ and $\rho$ we see that this is indeed a surjection. 
The kernel of this map is pairs $(a,b) \in P$ such that $\pi(a) = 0 = \rho(b)$, which is exactly $I \oplus J$.

The universal property of the pullback ensures that if (\ref{apws:acrossSeq}) and (\ref{apws:downSeq}) both split then (\ref{apws:pullSeq}) splits. 
On the other hand if we have a splitting from $\mathbb{C}$ to $P$, then simply composing that with the coordinate projections will show that (\ref{apws:acrossSeq}) and (\ref{apws:downSeq}) both split.
\end{proof}

\begin{remark}
In the form of a diagram we have shown that if we are given sequences (\ref{apws:acrossSeq}) and (\ref{apws:downSeq}) as in the above proposition, then the following diagram commutes and has exact rows, columns and diagonal.
\[
	\xymatrix{
	0 \ar[dr] & & 0 \ar[d] & 0 \ar[d] & \\
	 & I \oplus J \ar[dr] & J \ar[d] \ar@{=}[r] & J \ar[d] & \\
	0 \ar[r] & I \ar@{=}[d] \ar[r] & P \ar[d] \ar[dr] \ar[r] & B \ar[d] \ar[r] & 0 \\
	0 \ar[r] & I \ar[r] & A \ar[d] \ar[r] & \mathbb{C} \ar[d] \ar[dr] \ar[r] & 0 \\
	& & 0 & 0 & 0
	}
\]
\end{remark}

Now that we have a tool to construct non-split extensions from a split and a non-split one, we need a tool to tell us if the new extension is semiprojective. 
The following proposition is very slight generalization of \cite[Proposition 5.19]{apws:Neubuser} (where the ideal has to be the stabilization of a unital $C^*$-algebra). 
The proofs are essentially identical, but since \cite{apws:Neubuser} is in German, we include a short proof.

\begin{proposition} \label{apws:spTester}
Consider a short exact sequence 
\begin{equation*}
	0 \to I \to A \stackrel{\rho}{\to} Q \to 0.
\end{equation*}
If $I$ is generated as an ideal by finitely many projections and $A$ is semiprojective then $Q$ is semiprojective. 
\end{proposition}
\begin{proof}
Suppose we are given $B$, an increasing sequence of ideals $(J_k)$ in $B$, and a $*$-homomorphism $\phi \colon Q \to B / J$, where  $J = \overline{\cup_k J_k}$.
For all $k \in \mathbb{N}$, we let $\pi_{k,\infty} \colon B/J_k \to B/J$ be the natural quotient map.  
By the semiprojectivity of $A$ we can find and $n \in \mathbb{N}$ and a $*$-homomorphism $\psi \colon A \to B/J_n$ such that $\pi_{n,\infty} \circ \psi = \phi \circ \rho$. 

Let $p_1, p_2, \ldots, p_m$ be projections that generate $I$. 
For all $i$ we have $\rho(p_i) = 0$, and therefore we have $(\pi_{n,\infty} \circ \psi)(p_i) = 0$.
Hence, we can use \cite[Lemma 2.13]{apws:BlackadarShape} to deduce that there must be some $l \geq n$ such that $(\pi_{n,l} \circ \psi)(p_i) = 0$ for all $i=1,2,\ldots,m$. 
Since the $p_i$ generate $I$, we then have $(\pi_{m,k} \circ \psi)(I) = 0$, so $\pi_{n,l} \circ \psi$ drops to a $*$-homomorphism $\bar{\psi} \colon Q \to B/J_l$ with 
$\pi_{l,\infty} \circ \bar{\psi} = \phi$.
Thus $\bar{\psi}$ and $l$ combine to show that $Q$ is semiprojective. 
\end{proof}

Our strategy is now the following: Find a non-split short exact sequence 
\begin{equation*}
	0 \to J \to B \to \mathbb{C} \to 0,
\end{equation*}
such that $J$ has a full projection. 
We will then use the construction in Proposition \ref{apws:constructionNonSplit} on that and the Eilers-Katsura example, to produce a new non-split extension, which we can show, using Proposition \ref{apws:spTester}, has the desired properties. 

\section{Constructing a counterexample} \label{apws:sec:counter}

We begin this section by constructing a non-split short exact sequence where the ideal is semiprojective and contains a full projection, and the quotient is the complex numbers. 
To prove that the constructed sequence is non-split we will use $K$-theory. 
In particular, we will show that one of the boundary maps in the six-term exact sequence is non-zero.
Since $K_1(\mathbb{C}) = 0$, we need a semiprojective $C^*$-algebra with non-zero $K_1$-group. 
We will use a Kirchberg algebra. 

\begin{definition}
A separable, simple, nuclear, purely infinite $C^*$-algebras is called a \index{Kirchberg algebra} Kirchberg algebra. If it also satisfies the universal coefficient theorem, we call it a UCT Kirchberg algebra. 
\end{definition}

\begin{definition}
Denote by $\mathcal{P}_\infty$ the unital UCT Kirchberg algebra with $K_0(\mathcal{P}_\infty) = 0$ and $K_1(\mathcal{P}_\infty) = \mathbb{Z}$.
\end{definition}

Building on the work of Blackadar (\cite{apws:BlackadarSimpleSp}) and Szymanski (\cite{apws:Szymanski}), Spielberg has shown in \cite[Theorem 3.12]{apws:Spielberg} that any Kirchberg algebra with finitely generated $K$-theory and torsion free $K_1$-group is semiprojective.
In particular we have:  

\begin{theorem}[Spielberg] \label{apws:spielbergsGift}
Let $\mathbb{K}$ denote the algebra of compact operators. 
The Kirchberg algebra $\mathcal{P}_\infty \otimes \mathbb{K}$ is semiprojective. 
\end{theorem}

We can now construct a non-split sequence with a semiprojective ideal that contains a full projection. 

\begin{proposition} \label{apws:myFavoriteSeq}
There exists a non-split short exact sequence
\[
	0 \to J \to E \to \mathbb{C} \to 0,
\]
where $J$ is semiprojective and contains a full projection. 
\end{proposition}
\begin{proof}
Put $J = \mathcal{P}_\infty \otimes \mathcal{K}$, as the stabilization of a unital algebra $J$ contains a full projection.
By Theorem \ref{apws:spielbergsGift}, it is semiprojective. 
We will pick $E$ such that the boundary map in $K$-theory from $K_0(\mathbb{C})$ to $K_1(J)$ is non-zero.
Since $K$-theory is split exact this implies that the sequence does not split. 

We have the following short exact sequence:
\[
	0 \to J \to M(J) \to M(J)/J \to 0.
\]
If we let $\eta \colon K_0(M(J)/J) \to K_1(J)$ be the boundary map in the six-term exact sequence arising from the above extension, then by \cite[Proposition 12.2.1]{apws:BlackadarKBook} $\eta$ is an isomorphism. 
In particular 
\[
	K_0(M(J)/J) \cong K_1(J) \cong K_1(\mathcal{P}_\infty) = \mathbb{Z}. 
\]
By \cite[Theorem 2.2]{apws:LinZhang}, the corona algebra $M(J)/J$ has a continuous scale and so by \cite[Theorem 3.2]{apws:LinSimpleCorona} it is simple and purely infinite. 
Since $M(J)/J$ is also unital there is, by \cite[Corollary 6.11.8]{apws:BlackadarKBook}, a projection $p \in M(J)/J$ such that the class of $p$ in $K_0(M(J)/J)$ is $1 \in \mathbb{Z}$. 
Define a $*$-homomorphism $\tau \colon \mathbb{C} \to M(J)/J$ by $\tau(\lambda) = \lambda p$, and notice that $K_0(\tau)$ is an isomorphism of groups. 

Let $E = M(J) \oplus_{M(J)/J} \mathbb{C}$ where the pullback is taken over the quotient map from the multiplier algebra to the corona algebra and $\tau$.
We have the following commutative diagram which has exact rows (see \cite[Proposition 3.2.9]{apws:WeggeOlsenBook}):
\[
	\xymatrix{
		0 \ar[r] & J \ar[r] \ar@{=}[d]	& E \ar[r] \ar[d]	& \mathbb{C} \ar[r] \ar[d]^{\tau} & 0 \\
		0 \ar[r] & J \ar[r]				& M(J) \ar[r]		& M(J) / J \ar[r]			& 0
	}
\]   

Let $\delta$ denote the boundary map from $K_0(\mathbb{C})$ to $K_1(J)$ in the six-term exact sequence associated to the short exact sequence on top.
By \cite[Proposition 12.2.1]{apws:RordamKBook} the following square commutes:
\[
	\xymatrix{
		K_0(\mathbb{C}) \ar[r]^-\delta \ar[d]_{K_0(\tau)}	& K_1(J) \ar@{=}[d] \\
		K_0(M(J)/J) \ar[r]_-\eta						& K_1(J) 
	}
\]
Since $\eta$ and $K_0(\tau)$ are isomorphisms, we must have that $\delta$ is an isomorphism. 
In particular $\delta$ is non-zero, so the sequence 
\[
	0 \to J \to E \to \mathbb{C} \to 0
\]
does not split.
\end{proof}

We can now prove our main theorem.

\begin{theorem} \label{apws:mainThm}
There exists a non-split short exact sequence
\begin{equation*}
	0 \to K \to B \to \mathbb{C} \to 0,
\end{equation*}
such that $K$ is semiprojective but $B$ is not. 
\end{theorem}
\begin{proof}
Let 
\begin{equation*}
	0 \to I \to A \stackrel{\pi}{\to} \mathbb{C} \to 0
\end{equation*}
be a short exact sequence such that $I$ is separable and semiprojective but $A$ is not semiprojective, e.g. one of the extensions constructed by Eilers and Katsura (Theorem \ref{apws:EKseq}), and let 
\begin{equation}
	0 \to J \to E \stackrel{\rho}{\to} \mathbb{C} \to 0 \label{apws:helperSeq}
\end{equation}
be the non-split extension constructed in Proposition \ref{apws:myFavoriteSeq}. 

Put $B = A \oplus_\mathbb{C} E$ where the pullback is taken over $\pi$ and $\rho$.
By Proposition \ref{apws:constructionNonSplit} we have the following two short exact sequence:
\begin{eqnarray}
	0 \to I \oplus J \to &B& \to \mathbb{C} \to 0, \quad \mathrm{ and, } \\ \label{apws:theCounterExample}
	0 \to J \to &B& \to A \to 0. \label{apws:notSpSeq}
\end{eqnarray}
Furthermore (\ref{apws:theCounterExample}) does not split as (\ref{apws:helperSeq}) does not split. 

Since $J$ has a full projection and $A$ is not semiprojective Proposition \ref{apws:spTester} applied to (\ref{apws:notSpSeq}) gives us that $B$ is not semiprojective.
To complete the proof we put $K = I \oplus J$ and notice that $K$ is semiprojective, as it is the sum of two separable semiprojective $C^*$-algebras (\cite[Theorem 4.2]{apws:LoringDirectSum}). 
\end{proof}

\section{Acknowledgments}

The author thanks Takeshi Katsura and S{\o}ren Eilers for bringing the problem to his attention.
S{\o}ren Eilers is also thanked for giving discussions concerning the problem. 
This research was supported by the Danish National Research Foundation (DNRF) through the Centre for Symmetry and Deformation.

\end{document}